\documentclass[a4paper,fleqn]{cas-sc}
\usepackage[numbers]{natbib}

\def\tsc#1{\csdef{#1}{\textsc{\lowercase{#1}}\xspace}}
\tsc{WGM}
\tsc{QE}
\usepackage{framed}
\usepackage{amssymb}
\usepackage{latexsym}
\usepackage{amsfonts}
\usepackage{amsthm}
\usepackage{amsbsy}
\usepackage{amsmath}
\usepackage{graphics,graphicx}
\usepackage{subfigure}
\usepackage{multirow,multicol}
\usepackage{bm}
\usepackage{bbm}
\usepackage{color}
\usepackage{float}
\usepackage{tikz}
\usepackage{epstopdf}
\usepackage{epsfig}
\usepackage{algorithm}
\usepackage{algorithmic}
\usepackage{natbib}
\usepackage{listings} 

\usepackage{url}
\usepackage{xcolor}
\definecolor{newcolor}{rgb}{.8,.349,.1}

\newtheorem{theorem}{Theorem}
\newtheorem{lemma}[theorem]{Lemma}
\newdefinition{remark}{Remark}

\newdefinition{method}{Method}
\newdefinition{example}[theorem]{Example}

\numberwithin{equation}{section}
\NewDocumentCommand{\dgal}{sO{}m}{%
	\IfBooleanTF{#1}
	{\dgalext{#3}}
	{\dgalx[#2]{#3}}%
}

\NewDocumentCommand{\dgalext}{m}{%
	\sbox0{%
		\mathsurround=0pt 
		$\left\{\vphantom{#1}\right.\kern-\nulldelimiterspace$%
	}%
	\sbox2{\{}%
	\ifdim\ht0=\ht2
	\{\kern-.45\wd2 \{#1\}\kern-.45\wd2 \}%
	\else
	\fi
}

\NewDocumentCommand{\dgalx}{om}{%
	\sbox0{\mathsurround=0pt$#1\{$}%
	\sbox2{\{}%
	\ifdim\ht0=\ht2
	\{\kern-.45\wd2 \{#2\}\kern-.45\wd2 \}%
	\else
	\mathopen{#1\{\kern-.5\wd0 #1\{}
	#2
	\mathclose{#1\}\kern-.5\wd0 #1\}}
	\fi
}

\newcommand{\f}[1]{\mathbf{#1}}
\newcommand{\norm}[1]{\left\| #1 \right\|}
\newcommand{\abs}[1]{\lvert#1\rvert}

\newcommand{\lapla}{\Delta}
\newcommand{\manifold}{\mathcal{M}}

\usepackage{subcaption}

\usepackage{multirow}

\graphicspath{{fig/}{fig/teapot/}{fig/stern}{fig/elliot}{fig/spheremix}{fig/sphere}{fig/square}}

\begin{document}
\let\WriteBookmarks\relax
\def\floatpagepagefraction{1}
\def\textpagefraction{.001}
\let\printorcid\relax 

\shorttitle{Gradient enhanced ADMM for dynamic OT on surfaces}
\shortauthors{G. Dong, H. Guo, C. Jiang and Z. Shi}

\title[mode = title]{Gradient enhanced ADMM Algorithm for dynamic optimal transport on surfaces}  

%
%
\author[1]{Guozhi Dong}
\address[1]{School of Mathematics and Statistics, HNP-LAMA,  Central South University, Changsha 410083, China.}
\ead{guozhi.dong@csu.edu.cn}

\author[2]{Hailong Guo}
\address[2]{School of Mathematics and Statistics,  The University of Melbourne,  Parkville, VIC 3010, Australia.}
\ead{hailong.guo@unimelb.edu.au}

\author[3]{Chengrun Jiang}
\address[3]{Department of Mathematical Sciences, Tsinghua University, Beijing, 100084,  China.}
\ead{jcr22@mails.tsinghua.edu.cn}

\author[4,5]{Zuoqiang Shi}
\address[4]{Yau Mathematical Sciences Center, Tsinghua University, Beijing, 100084, China.}
\address[5]{Yanqi Lake Beijing Institute of Mathematical Sciences and Applications, Beijing, 101408, China.}
\ead{zqshi@tsinghua.edu.cn}

\begin{abstract}
A gradient enhanced ADMM algorithm for optimal transport on general surfaces is proposed in this paper. 
Based on Benamou and Brenier's dynamical formulation, we combine gradient recovery techniques on surfaces with the ADMM algorithm, not only improving the computational accuracy, but also providing a novel method to deal with dual variables in the algorithm. 
This method avoids the use of stagger grids, has better accuracy and is more robust comparing to other averaging techniques.
\end{abstract}




	
	
	
	

\maketitle


\section{Introduction}
Optimal transport (OT) as a research topic has been historically formulated by Monge.  Initially, it was a problem on finding ways of distributing piles of sand from one place to another with ``least cost" in transportation.
Although it was an interesting problem, which has been continuously receiving attentions from many different aspects, perhaps Monge never imagined that it has been so heavily influencing the shape of modern mathematics.
Many efforts have been devoted to the research of optimal transport.
The Monge's problem had been reformulated by Kantorovich \cite{Kantorovich42} among others as logistic and economical problems, which make optimization tools, e.g. linear programming been firmly developed.
In the end of 80s in the last century, it was Brenier who proved the equivalence of the Monge's problem and the Kantorovich's reformulation for the continuous case under certain conditions \cite{Bre89}.
Since then, optimal transport has found deep connections to several fields in pure and applied mathematics which were far away before. Understanding of these types of problems have made optimal transport a quite diverse concept but a cluster topic in mathematics \cite{Villani09,FigGla23}.
In the late 90s, a so-called dynamic formulation of OT has been proposed by Benamou and Brenier \cite{BenBre00}, which shows an even richer perspective in connections to fluid dynamics and optimization with partial different equation constraints. Nowadays, OT has made important impact in many branches of pure, applied and computational mathematics, particularly, it provides intersections of mathematics with other subjects, like economy, engineering, computer science  \cite{seismic-engquist,WGAN}. On the application side, OT can be applied to image registration, computational geometry, machine learning and so on \cite{PeyCut19}. 

In this paper, we study the computational methods for OT when the probability measures are defined on general surfaces or manifolds denoted by $\mathcal{M}$. Theoretical investigation of OT on manifolds was done in \cite{FelMcC02,FatFig10}, while the computational methods is a quite recent topic \cite{LCCS2018,GraAiJ23,yu2023meanfieldmanifold}.
We pick up the perspective of the dynamic formulation for OT attributed to Benamou and Brenier. One of the feature of the dynamical approach is that it provides the trajectory of the evolution of density function from the initial state to the target state, which is considered to be the geodesic on the Wasserstein manifold of probability measures. 
Precisely, we study numerical algorithms for the following partial differential equation constrained optimization problems:
\begin{eqnarray}\label{equ:ot_dynamic_convex}
	W_2^2(\rho_0,\rho_1):= &\underset{\rho,\mathbf{m}}{\inf}\int_{0}^{1}\int_{\mathcal{M}}\frac{|\mathbf{m}(t,\mathbf{x})|^2}{2\rho(t,\mathbf{x})} d\mu dt, \nonumber \\
	\text{ subject to }& \rho(0,\cdot)=\rho_0, \quad \rho(1,\cdot)=\rho_1, \quad \rho \geq 0, \nonumber\\
	& \partial_t \rho+ \text{div}_g (\mathbf{m})  =0.\label{equ:pde1} 
\end{eqnarray}
Here $\rho:[0,1]\times \mathcal{M} \to \mathbb{R}^+$ is the density function of the distribution, and $\mathbf{m}=\rho \mathbf{v}$ where $\mathbf{v}:[0,1]\times \mathcal{M}\to \mathcal{TM}$ is the velocity vector field which describes the evolution of $\rho$.
When $\mathcal{M}$ is a planar subset in Euclidean space, it has been studied intensively in the literature.
Some more background of this problem can be found in Section \ref{ssec:BB_form}, as well as the seminal paper of Benamou and Brenier \cite{BenBre00}, in which case, some deep learning method is proposed recently \cite{WanZhaBaoDonShi23}.
In contrast, the numerical investigation for dynamical OT in surface settings is still limited \cite{LCCS2018, yu2023meanfieldmanifold}.

One popular algorithm for the above problem is the alternating directional multiplier method (ADMM) which has been quite intensively investigated in the literature, see for instance \cite{BoyPar10} and the references therein. In fact, ADMM algorithm is usually applied to the dual problem of \eqref{equ:ot_dynamic_convex} in order to fit the required setting of ADMM in terms of convergence.
Another difficulty often omitted in the literature is that ADMM algorithm for problems in infinite dimensional spaces is more delicate to show its convergence, for which we refer to \cite{AttBolRedSou08}.
Our goal in the current work is to develop robust and efficient numerical solver for ADMM algorithm in the case of $\mathbf{x}\in \mathcal{M}$ where $\mathcal{M}$ is a $d-$dimensional Riemannian manifold.
ADMM algorithm requires alternatively updating primal and dual variables, while in the setting of dynamical OT, dual variables involve the spatial and temporal derivatives of their primal variables. In the surface setting, derivatives require the knowledge of the metric of surface or the normal vectors, which due to the discretization of the surface, are more difficult to evaluate than in the planar setting.
Moreover, discretization of a surface also destroys the differential structures of the surface, in which case, how to compute derivatives of functions on such discrete surfaces is an issue.
On the other hand, spatial discretization, e.g. using finite element method or finite difference method,  may require stagger grids to have the consistency of the primal and dual variables as done in \cite{PGOu2014}, or
use some transform between piecewise constant function and piecewise linear function \cite{LCCS2018}.

The key step in the ADMM solver for dynamical OT is to solve a time-space Poisson equation with right hand side function involving divergence of some vector field. From the numerical perspective in solutions of PDEs, e.g., finite element solutions, when we take gradients or divergences, they often lose one order accuracy in comparing to the solutions themselves. Closer inspection of the ALG2 algorithm in \cite{BenBre00}, we find that it involves a consecutive process of taking gradient and then taking divergence. Theoretically, when linear finite element method or standard five-point finite difference method is applied, we would not expect any order of convergence. However, extensive numerical results in the literature, like \cite{BenBre00, PGOu2014, yu2024fastproximal}, to just name a few, show the convergence of the numerical results. 

This paper provides a new angle to understand the hidden mechanism of the convergence, and proposing a numerical scheme which is easier to implement and more robust in performance. The key ingredient to obtain convergent numerical results lies in the process of interpolating the piecewise constant into piecewise linear counterpart by averaging quantities of neighboring grids. In post-processing of numerical data, e.g., finite element solutions, such a simple averaging technique is a popular gradient recovery method, which was originally proposed by Zienkiewicz and Zhu for the purpose of adaptive computation \cite{ZZ1992}. The simple averaging technique was proved to achieve $\mathcal{O}(h^2)$ on some meshes like uniform meshes of regular type \cite{xu2004super}. In other cases, differentiation also achieves $\mathcal{O}(h)$ order convergence which explains why it is able to observe convergence in the works from the literature. However, the simple averaging gradient recovery is not able to achieve $\mathcal{O}(h^2)$ even on some uniform meshes, e.g. Chevron mesh \cite{DongGuo2020}. To overcome such drawback, better techniques are proposed. In particular, the polynomial preserving recovery is proven to be $\mathcal{O}(h^2)$ order consistent for all types of meshes. Such kind robust techniques has been generalized to the surface or manifold setting in \cite{DongGuoGuo2024}. In this paper, together with gradient recovery techniques on manifolds, we propose gradient enhanced ADMM method. The (surface) gradient recovery serves two purposes: one is to improve the accuracy of numerical gradient/divergence  and the other is to facilitate the implementation on a single mesh where generating dual mesh on surfaces is nontrivial. In the end, we propose a fast algorithm for dynamic OT on manifolds without staggered grid, which ensures the consistency of the primal and dual variables using the same grid simultaneously. In particular, we show the efficiency and robustness of our method with numerical examples.


The structure of the paper is outlined as follows:
Section \ref{sec:pre} collects some brief introduction on the dynamical formulation of OT, and gradient recovery techniques, particularly PPPR.
Section \ref{sec:algo} presents the proposed algorithm in detail, while their numerical results are documented in Section \ref{sec:num}.

\section{Preliminary}\label{sec:pre}
We shall briefly summarize the Benamou--Brenier formulation of OT, and the augmented Lagrange multiplier of its dual form, as well as the gradient recovery schemes we use in this work. First of all, let us describe the setting of the geometry and notations.

The manifold $\mathcal{M}$ we consider is a $d$-dimensional closed Riemannian manifold embedded in $\mathbb{R}^m$ for $d+1\leq m$, which endowed with the Riemann metric $g$.
Using local coordinates, the gradient of functions on the manifold is represented as follows
\begin{equation}
	\label{eq:local_gradient1}
	\nabla_g u= \sum_{i,j} g^{ij}\partial_j u \partial_i,
\end{equation}
where $g^{ij}$ is the entry of the inverse of the metric tensor $g$, and $\partial_i$ denotes the tangential basis.
In practice, the above formula can be realized using local parametrization $\f r:\Omega\rightarrow S \subset \manifold$.
Then we rewrite \eqref{eq:local_gradient1} via this map:
\begin{equation}
	\label{eq:local_gradient2}
	(\nabla_g u)\circ \f r=\nabla \bar{u}  (g\circ \f r)^{-1} \partial \f r.
\end{equation}
Here $\bar{u}=u\circ \mathbf{r}$ is the pullback of function $u$ to the local parametric domain $\Omega$, $\nabla$ denotes the gradient operator on $\Omega$, $\partial \f r$ is the Jacobian of $\f r$, and 
\[g\circ \f r =\partial \f r(\partial \f r)^\top.\]
In the light of \eqref{eq:local_gradient1}, the Laplace-Beltrami operator can be represented as
\begin{equation}
	\label{eq:laplace_beltrami}
	\lapla_g u=\text{div}_g(\nabla_g u)=\frac{1}{\sqrt{\abs{\det g}}}\partial_i(g^{ij}\sqrt{\abs{\det g}}\partial_j u).
\end{equation}

\subsection{Benamou--Brenier formulation of OT}
\label{ssec:BB_form}
We study the following problem to calculate the $L^2$ Wasserstein distance between two probability measure $\rho_0\in P(\mathcal{M})$ and $\rho_1\in P(\mathcal{M})$:
\begin{equation}\label{equ:ot_manifold}
	W_2^2(\rho_0,\rho_1):=\inf_{\Pi}\int_{\mathcal{M}}\frac{1}{2}|\Pi(\mathbf{x})-\mathbf{x}|^2\rho_0(\mathbf{x})d\mu ,
\end{equation}
where $\Pi:\mathcal{M} \to \mathcal{M}$ is a transport map which moves particle at location $\mathbf{x}$ to $y=\Pi(\mathbf{x})$. Note that both $\rho_0$ and $\rho_1$ are characterized by their density functions, still denoted by $\rho_0(\mathbf{x})$ and $\rho_1(\mathbf{x})$, respectively,
which tell the distribution of the particles before and after transportation, respectively.

How to compute the optimal transport map has been a challenging problem since Monge.
Due to the seminal work of Benamou and Brenier\cite{BenBre00}, the Wasserstein distance in \eqref{equ:ot_manifold} can be equivalently calculated via solving a dynamical flow constrained optimization problem
\begin{equation}\label{equ:ot_dynamic}
	\begin{aligned}
		W_2^2(\rho_0,\rho_1):= &\underset{\rho,\mathbf{v}}{\inf}\int_{0}^{1}\int_{\mathcal{M}}\frac{1}{2}|\mathbf{v}(t,\mathbf{x})|^2\rho(t,\mathbf{x})d\mu dt,\\
		\text{ subject to }& \rho(0,\cdot)=\rho_0, \quad \rho(1,\cdot)=\rho_1, \quad \rho \geq 0,\\
		& \partial_t \rho+ \text{div}_g(\rho \mathbf{v})=0.
	\end{aligned}
\end{equation}
Here $\mathbf{v}: [0,1]\times \mathcal{M} \to \mathcal{TM}$ is a temporal-spatial tangent vector field, and $\rho:[0,1]\times \mathcal{M} \to \mathbb{R}^+$ defines an interpolation curve between the two densities $\rho_0$ and $\rho_1$. In this paper, we use $ \mathcal{P}:=[0,1] \times \mathcal{M}$ to denote the temporal-spatial product manifold.
In the following algorithmic development, one sorts numerical solutions of partial differential equations on this manifold, and also some gradient recovery techniques on this non-Euclidean domain.

Let $\mathbf{m}=\rho \mathbf{v}$, by a substitution of variable, the problem in \eqref{equ:ot_dynamic} can be reformulated as a convex variational problem subject to linear partial differential equation constraint as in \eqref{equ:ot_dynamic_convex}.
Note here if we name $\boldsymbol{\sigma}=(\rho,\mathbf{m})$, then the PDE \eqref{equ:pde1} can be written equivalently as $\mbox{div}_{\mathbf{p}} \boldsymbol{\sigma} =0$, where $\mbox{div}_{\mathbf{p}}$ is the divergence operator for functions defined on the temporal-spatial product space $\mathcal{P}$. Later we shall use $\nabla_{\mathbf{p}}$ to denote the temporal-spatial gradient operator on $\mathcal{P}$,  and $\Delta_{\mathbf{p}}$ to denote the temporal-spatial Laplacian operator, i.e.
\begin{equation}
	\nabla_{\mathbf{p}} u=(\partial_t u,\nabla_{g}u)^T\quad\text{and}\quad\Delta_{\mathbf{p}}u=\partial_{tt}u+\Delta_{g}u.
	\nonumber
\end{equation}

For simplicity, we define
\[K(\rho,\mathbf{m}):= \int_{0}^{1}\int_{\mathcal{M}}\frac{|\mathbf{m}(t,\mathbf{x})|^2}{2\rho(t,\mathbf{x})} d\mu dt.\]
Here we first give the Lagrange multiplier to formulate a saddle point problem:
\begin{equation}\label{equ:LM}
	\underset{\rho,\mathbf{m}}{\inf}\;\underset{u,\lambda \leq 0}{sup} L(\rho,\mathbf{m},u,\lambda):=  K(\rho,\mathbf{m}) + \int_{0}^{1}\int_{\mathcal{M}} (\partial_t \rho u  +   \nabla_g \mathbf{m} u + \lambda \rho ) d\mu dt.
\end{equation}
Taking into account the initial conditions, formally we derive the first-order stationary condition of \eqref{equ:LM} as follows
\begin{equation}\left\{
	\begin{aligned}
		\frac{\mathbf{m}}{\rho}-\nabla_g u =0, & \quad \partial_t \rho + \text{div}_g \mathbf{m} =0,\\
		-\frac{\abs{\mathbf{m}}^2}{2\rho^2}-\partial_tu +\lambda=0, & \quad \rho(0,\cdot)=\rho_0, \quad \rho(1,\cdot)=\rho_1, \quad \lambda \leq 0.
	\end{aligned}\right.
\end{equation}

\subsection{Gradient recovery on manifolds}\label{sec:gr}
Gradient recovery for numerical solutions in planar domain is a well-developed topic, there are many mature methods proposed in the literature, see for instance \cite{Zhang2007, ZZ1992, Zhang2005new}.  While gradient recovery on surfaces is relatively new, where some recent papers on this topic can be found in e.g. \cite{WeiChenHuang09, DongGuo2020,DongGuoGuo2024}.
We will first review the idea of a gradient recovery method called polynomial preserving recovery(PPR), and then to discuss generalization of this method to the surface setting, which is the parametric polynomial preserving recovery(PPPR).
\subsubsection{Polynomial preserving recovery}\label{ssec:ppr}
To lay the ground, we introduce the PPR gradient recovery for time discretization. Let $I_{\tau} = \cup_{j=0}^{N_t-1}(t_j, t_{j+1})$ be a partition of the interval $I = (0, 1)$  with $\tau_j = t_{j+1}-t_j$ and $S_{\tau}$ be continuous finite element space on $I_{\tau}$, i.e.
\begin{equation}\label{equ:timefem}
	S_{\tau} = \left\{v_{\tau}\in C^0([0, 1]):  v_{\tau}|_{(t_j, t_{j+1})}
	\in \mathbb{P}_1(t_j, t_{j+1})\right\}.
\end{equation}
For any $v_{\tau}\in S_{\tau}$, $\nabla v_{\tau}$ is piecewise constant. The main idea of gradient recovery is to  smooth the piecewise gradient into a continuous piecewise linear function. Let $G_{\tau}: S_{\tau} \rightarrow S_{\tau}$ denote the PPR gradient operator \cite{Zhang2007} on $I_{\tau}$. Suppose $\{\phi_{j}(t)\}$ is the set of nodal basis functions of $S_{\tau}$ and we can write $G_{\tau} v_{\tau} = \sum_{j=0}^{N_{t}}(G_{\tau}v_{\tau})(t_j)\phi_j(t)$.  It is sufficient to define the value of $G_{\tau}(v_{\tau})(t_j)$. Define 
\begin{eqnarray}
	I_{t_j} = \left\{
	\begin{array}{ll}
		(t_0, t_2), &  j = 0; \\
		(t_{j-1}, t_{j+1}), & 1 \le j \le N_t -1; \\
		(t_{N_t-2}, t_{N_t}), & j = N_t. 
	\end{array}
	\right.
\end{eqnarray}
We fit a quadratic polynomial $p_{t_j}$ in the following least-squares sense
\begin{equation}
	p_{t_j}=\arg \min _{p \in \mathbb{P}^2\left(I_{t_j}\right)} \sum_{t_k \in I_{t_j} \cap \mathcal{N}_{\tau}}\left|p\left(t_k\right)-v_{\tau}\left(t_k\right)\right|^2,
\end{equation}
where $N_{\tau} = \{t_0, t_1, \cdots, t_{N_t}\}$.
Then the recovered gradient $(G_{\tau} v_{\tau})(t_j)$ is defined as
\begin{equation}
	(G_{\tau} v_{\tau})(t_j)  = \frac{d p_{t_j}(t)}{dt}|_{t = t_j}. 
\end{equation}
It is easy to see that $(G_{\tau} v_{\tau})(t_j)$ is a second-order finite difference scheme at $t_j$. 
As proved in \cite{Zhang2007}, the recovery operator $G_{\tau}$ satisfies the following consistency results
\begin{equation}\label{equ:timeconsistency}
	\|G_{\tau}v - \frac{dv}{dt}\|_{H^1(0, 1)} \lesssim \tau^2 \|v\|_{H^3(0, 1)}. 
\end{equation}
\subsubsection{Parametric polynomial preserving recovery}\label{ssec:PPPR}
From the parametric form \eqref{eq:local_gradient2} of surface gradient $\nabla_g$, it informs that one can compute the surface gradient of a function $u:\mathcal{M}\to \mathbb{R}$ via the multiplication of two functions: One being the gradient of the function $\bar{u}=u\circ \mathbf{r}$ and the other being the generalized inverse of the Jacobian of the parametric map, which is $(\partial \mathbf{r})^\dag = (\partial \mathbf{r} (\partial \mathbf{r})^\top )^{-1} \partial \mathbf{r} $. Both functions are defined on the same local Euclidean domain. Based on this observation, Dong and Guo \cite{DongGuo2020} proposed the parametric polynomial preserving recovery (PPPR) by replacing the gradient/Jacobian in the Euclidean space with their recovered counterparts using parametric polynomial preserving recovery. Let $\bar{G}_h$ denote the polynomial preserving recovery in the local Euclidean space as defined in \cite{Zhang2005new}. Then, the PPPR gradient recovery operator at vertex $\xi_i$ can be represented as
\begin{equation}
	G_hu_h(\xi_i) = \bar{G}_h \bar{u}_h(\xi_i)(\bar{G}_h \mathbf{r}_h(\xi_i))^{\dag},
\end{equation}
where $\bar{u}_h$ is the local pushback of $u_h$ onto the local parametric domain, and $\mathbf{r}_h$ is the local approximation of the parametric map. For the PPPR operator $G_h$, \cite{DongGuo2020} shows that it has second-order consistency in the following sense:
\begin{equation}
	\left\|\nabla_g u-\left(T_h\right)^{-1} G_h u_I\right\|_{0, \mathcal{M}} \leq h^2 \sqrt{\mathcal{A}(\mathcal{M})} D\left(g, g^{-1}\right)\|u\|_{3, \infty, \mathcal{M}},
\end{equation}
where $\mathcal{A}(\mathcal{M})$ is the area measure of the manifold $\mathcal{M}$ and $D(g, g^{-1})$ is some constant depends on the metric tensor $g$.


\section{Gradient enhanced ADMM algorithm}\label{sec:algo}
In this section,  we shall present a gradient enhanced algorithms based on ALG2 \cite{BenBre00} in the manifold setting.
\subsection{ADMM algorithm for dynamical OT}
In \cite{BenBre00}, an alternating direction method of multipliers (ADMM) algorithmic framework for numerical solutions of dynamical optimal transport were introduced.
However, a direct development of an ADMM algorithm building on the Lagrange formulation in \eqref{equ:LM} appears to be nontrivial due to the first term which mixes $\rho$ and $\mathbf{m}$. To overcome this obstacle, the dual function of the first term in \eqref{equ:LM} has been considered. 
This is done by using the following result
\begin{lemma}\label{lem:dual_form}
	For any $\rho\in \mathbb{R}^+$, and $\mathbf{m}\in \mathbb{R}^d$, we have 
	\[ \frac{|\mathbf{m}|^2}{2\rho} = \sup_{(a,\mathbf{b})\in A} \{a\rho + \mathbf{b}\cdot \mathbf{m} \}, \]
	where $A$ is defined as 
	\[ A:= \{ (a,\mathbf{b}) \in \mathbb{R}\times \mathbb{R}^d:  a+ \frac{|\mathbf{b}|^2}{2}\leq 0 \} .\]
\end{lemma}
Let $\mathbf{q}=(a,\mathbf{b})$, and $\boldsymbol{\sigma}=(\rho,\mathbf{m})$, with the relation in Lemma \ref{lem:dual_form}, then the Lagrangian problem in \eqref{equ:LM} can be reformulated as
\begin{equation*}
	\begin{aligned}
		\inf_{\rho,\mathbf{m}} \sup_{u,\lambda<0} L(\rho,\mathbf{m},u,\lambda)=\sup_{u,\mathbf{q}\in A} \inf_{\boldsymbol{\sigma}} 
		&\int_{0}^{1}\int_{\mathcal{M}} (\boldsymbol{\sigma} \cdot \mathbf{q} -    [\partial_t u  ,   \nabla_g u]\cdot  \boldsymbol{\sigma} ) d\mu dt \\
		& - \int_{\mathcal{M}} (u(0,\cdot)\rho_0 -u(1,\cdot)\rho_1 )d\mathbf{x} .
	\end{aligned}
\end{equation*}
After some rearrangement, we have a new saddle-point problem that can be seen as a dual problem of the original Lagrange multiplier problem in \eqref{equ:LM}, which is
\begin{equation}\label{equ:re_LM}
	\sup_{u,\mathbf{q}\in A} \inf_{\boldsymbol{\sigma}} 
	\int_{0}^{1}\int_{\mathcal{M}} (\boldsymbol{\sigma}\cdot (\mathbf{q} -  \nabla_{\mathbf{p}}   \ u) )d\mu dt - \int_{\mathcal{M}} (u(0,\cdot)\rho_0 -u(1,\cdot)\rho_1 )d\mathbf{x} .
\end{equation}
Note that $\nabla_{\mathbf{p}}$ denotes the temporal-spatial gradient defined on the product manifold $\mathcal{P}$.
We introduce the notations for simplification of presentation
\begin{eqnarray}
	F(\mathbf{q}) & :=& \left\{\begin{aligned}
		&0, && \quad \mathbf{q}\in A,\\
		&+\infty, && \quad \text{else}.
	\end{aligned}\right.\\ 
	G(u)& := &\int_{\mathcal{M}} (u(0,\cdot)\rho_0 -u(1,\cdot)\rho_1 )d\mathbf{x} .  
\end{eqnarray}
Then the new saddle-point problem \eqref{equ:re_LM} is of the following new Lagrange formulation with a new multiplier $\boldsymbol{\sigma}$, while $\mathbf{q}$ and $u$ are primal variables:
\begin{equation}\label{equ:new_LM}
	\inf_{\rho,\mathbf{m}} \sup_{u,\lambda<0} L(\rho,\mathbf{m},u,\lambda)= -\inf_{u, \mathbf{q} } \sup_{\sigma}  F(\mathbf{q}) + G (u) + \int_{0}^{1}\int_{\mathcal{M}} (\boldsymbol{\sigma}\cdot (\nabla_{\mathbf{p}} u -\mathbf{q} ) )d\mu dt.
\end{equation}
Note that the new problem can be considered a Lagrange multiplier of equality constraint $\nabla_{\mathbf{p}} u =\mathbf{q}$.
In what follows, we shall then study this new Lagrange formulation and denote 
\begin{equation}\label{equ:new_LM2}
	L(\mathbf{q},u,\boldsymbol{\sigma}):= F(\mathbf{q}) + G (u) + \int_{0}^{1}\int_{\mathcal{M}} (\boldsymbol{\sigma} \cdot (\nabla_{\mathbf{p}} u - \mathbf{q} ) )d\mu dt.
\end{equation}
The prominent property of this new formulation is that $F$ and $G$ are functions of $\mathbf{q}$ and $u$, respectively, which fit into the framework of ADMM algorithms.
For such a development, an augmented Lagrange multiplier based on \eqref{equ:new_LM} is provided here
\begin{equation}\label{equ:new_ALM}
	L_r(\mathbf{q},u,\boldsymbol{\sigma}):= F(\mathbf{q}) + G (u) + \int_{0}^{1}\int_{\mathcal{M}} (\boldsymbol{\sigma} \cdot (\nabla_{\mathbf{p}} u -\mathbf{q} ) )d\mu dt  + \frac{r}{2} \norm{\nabla_{\mathbf{p}} u -\mathbf{q}}_{L^2(\mathcal{P})}^2.
\end{equation}
ADMM algorithms for approaching solutions of the form in \eqref{equ:new_ALM} in the Euclidean spaces have been intensively investigated after the seminal work of Benamou and Brenier, while only recently it has been generalized for probability measure on surfaces or general manifolds \cite{LCCS2018, yu2023meanfieldmanifold}.

For the convenience of numerical treatment, let us introduce some Hilbert spaces setting, to denote $\mathcal{C}\subset H(\mbox{div}_{\mathbf{p}},\mathcal{P})$ being the set of measures satisfying $\mbox{div}_{\mathbf{p}}\boldsymbol{\sigma}=0$.

The purpose of the gradient-enhanced algorithm is not only to facilitate the implementation in a single mesh but also to significantly improve the approximation accuracy of the right-hand side gradient function.
Now we provide a skeleton of the algorithm which is often referred as ALG2.
It is an ADMM algorithm for solving \eqref{equ:new_ALM} in a function space setting.
\begin{algorithm}\caption{ALG2 in \cite{BenBre00}}
	Initial guess: $\mathbf{q}^0$, $\boldsymbol{\sigma}^1$. For k=1, 2, 3, $\ldots$
	\begin{itemize}
		\item[Step 1] Find $u^k$ which solves
		\[ u^k=\operatorname*{arginf}_{u} L_r(\mathbf{q}^{k-1},u,\boldsymbol{\sigma}^{k}). \]
		\item[Step 2] Find $q^k$ which solves
		\[ \mathbf{q}^k=\operatorname*{arginf}_{\mathbf{q}} L_r(\mathbf{q},u^{k},\boldsymbol{\sigma}^{k}). \]
		\item[Step 3] Update $\boldsymbol{\sigma}^k$ as follows for some fixed parameter $\alpha_r\in \mathbb{R}^+$
		\[ \boldsymbol{\sigma}^{k+1}=\boldsymbol{\sigma}^k + \alpha_r (\nabla_{\mathbf{p}} u^k - \mathbf{q}^k). \]
	\end{itemize}
	\label{alg:ADMM}
\end{algorithm}
Some details of the first two steps are given as follows.
For Step 1, we write down the first-order optimality condition:
\begin{equation} 
	G(h)+ \langle  r\nabla_{\mathbf{p}} \cdot \mathbf{q}^{k-1} - \Delta_{\mathbf{p}} u^k - \nabla_{\mathbf{p}} \cdot \boldsymbol{\sigma}^k ,h  \rangle=0,\quad \text{ for all }\quad  h\in H(\mbox{div}_{\mathbf{p}},\mathcal{P}),
\end{equation}
which leads to the following time-space  equation
\begin{equation}\label{equ:timespace}
	\left\{
	\begin{aligned}
		-\Delta_{\mathbf{p}} u^k &=\nabla_{\mathbf{p}}\cdot  \sigma^k - r\nabla_{\mathbf{p}} \cdot \mathbf{q}^{k-1}, \\
		r\partial_t u^k(0;\cdot) &=\rho_0-\rho^k(0,\cdot)+ r a^{k-1}(0,\cdot), \\
		r\partial_t u^k(1;\cdot) &=\rho_1-\rho^k(1,\cdot)+ r a^{k-1}(1,\cdot), \\
		\int_0^1\int_{\mathcal{M}}& u^k(t,x) d\mu dt  = 0.
	\end{aligned}\right.
\end{equation}

In Step 2, we rewrite the optimization problem to have
\begin{equation}
	\begin{aligned}
		\inf_\mathbf{q}   L_r(\mathbf{q},u^{k},\boldsymbol{\sigma}^{k})& =\inf_{\mathbf{q}} F(\mathbf{q}) + \langle \boldsymbol{\sigma}^k , (\nabla_{\mathbf{p}} u^{k} - \mathbf{q} ) \rangle + \frac{ r}{2} \norm{\nabla_{\mathbf{p}} u^k -\mathbf{q}}_{L^2(\mathcal{P})}^2\\
		& =\inf_{\mathbf{q}\in A} \frac{2}{r}\left( \frac{r}{2}\langle \boldsymbol{\sigma}^k , (\nabla_{\mathbf{p}} u^{k} - \mathbf{q}) \rangle + (\frac{r}{2})^2 \norm{\nabla_{\mathbf{p}} u^k - \mathbf{q}}_{L^2(\mathcal{P})}^2  \right)\\
		& =\inf_{\mathbf{q}\in A}  \frac{2}{r} \| \frac{r}{2}(\nabla_{\mathbf{p}} u^{k} -\mathbf{q} )+ \frac{1}{2} \boldsymbol{\sigma}^k \|_{L^2(\mathcal{P})}^2 - \frac{1}{2r}\|  \boldsymbol{\sigma}^k \|_{L^2(\mathcal{P})}^2 .
	\end{aligned}
\end{equation}
This shows that if $\nabla_{\mathbf{p}} u^k+\frac{1}{r}\boldsymbol{\sigma}^k\in A$, then we have $\mathbf{q}=\nabla_{\mathbf{p}} u^k+\frac{1}{r}\boldsymbol{\sigma}^k$ to be the solution, otherwise, one can introduce a multiplier $\lambda\geq 0$ to take care of the inequality constraints. In that case we have
\begin{equation}\label{equ:projection}
	\mathbf{q}^k=(a^k,\mathbf{b}^k)=\left\{
	\begin{aligned}
		&(\alpha^k,\boldsymbol{\beta}^k)  &&    \quad  \text{ if }   \quad   \nabla_{\mathbf{p}} u^k+\frac{1}{r}\boldsymbol{\sigma}^k \in A; \\
		&\left(\alpha^k-\lambda, \frac{\boldsymbol{\beta}^k}{1+\lambda}\right) &&\quad  \text{else},
	\end{aligned}\right.
\end{equation}
where $(\alpha^k,\boldsymbol{\beta}^k)=  \nabla_{\mathbf{p}} u^k+\frac{1}{r}\boldsymbol{\sigma}^k$. Therefore, the key to this step is to calculate the value of the Lagrange multiplier $\lambda$.
Using the complementarity condition for the multiplier, we have 
\[ \lambda\left(a^k+\frac{|\mathbf{b}^k|^2}{2}\right)=0 \quad \text{ when } \quad \nabla_{\mathbf{p}} u^k+\frac{1}{r}\boldsymbol{\sigma}^k \notin A, \]
which leads to the following polynomial equation of $\lambda$
\begin{equation}\label{equ:cubic}
	\alpha^k-\lambda +\frac{|\boldsymbol{\beta}^k|^2}{2(1+\lambda)^2}=0 \quad \Rightarrow  \quad  -\lambda^3+ (\alpha^k-2)\lambda^2+ (2\alpha^k-1)\lambda +\alpha^k +\frac{|\boldsymbol{\beta}^k|^2}{2}=0 .
\end{equation}
The following lemma shows that the above polynomial function \eqref{equ:cubic} has one real root in the range $\lambda \geq 0$.
\begin{lemma}
	Given the relations in \eqref{equ:projection}, the cubic polynomial equation in \eqref{equ:cubic} admits a solution in the range $\lambda\geq 0$.
\end{lemma}
\begin{proof}
	Define $f(x)= -x^3+ (\alpha^k-2)x^2+ (2\alpha^k-1)x +\alpha^k +\frac{|\boldsymbol{\beta}^k|^2}{2}$. By taking the derivative of $f$, we find that $x=-1$ and $x=\frac{2\alpha^k-1}{3}$ are the two roots of $f'(x)$. This shows that they are the two local stationary points of $f(x)$. We have $f(-1)=\frac{|\boldsymbol{\beta}^k|^2}{2}>0$, $f(\frac{2\alpha^k-1}{3})=4(\frac{\alpha^k+1}{3})^3 +\frac{|\boldsymbol{\beta}^k|^2}{2}$. In particular $f(0)=\alpha^k +\frac{|\boldsymbol{\beta}^k|^2}{2}\geq 0$ when $\lambda\geq 0$, $\lim_{x\to +\infty} f(x)=-\infty$. This tells that the third-order polynomial equation $f(x)=0$ at least admits a solution for $x \geq 0$.
\end{proof}

The solution of \eqref{equ:cubic} can be calculated using either explicit formula or Newton's iteration.

\subsection{Discretization of the time-space Poisson equation}
In this subsection, the time-space Poisson equation is discretized using the coupled finite difference and finite element method.
To simplify the notation, we consider the following time-space Poisson equation
\begin{equation} \label{equ:tsharmonic}
	-\partial_{tt} u - \Delta_{g} u = \nabla_{\mathbf{p}} \cdot \mathbf{f}, \quad \text{in}\quad (0,1) \times \mathcal{M},
\end{equation}
where $\Delta_{g}$ represents the Laplace-Beltrami operator on the manifold $\mathcal{M}$ with metric $g$. The equation is subject to Neumann boundary conditions:
\begin{equation}\label{equ:nbc}
	\partial_t u(0, \cdot) = u^0(\cdot) \quad \text{and} \quad \partial_{t} u(1, \cdot) = u^1(\cdot)
\end{equation}
and it also satisfies the compatibility condition:
\begin{equation}
	\int_0^1 \int_{\mathcal{M}} u d\mu dt = 0.
\end{equation}

We firstly consider the discretization in time using the central finite difference methods. For such purpose, we divide the interval $[0, 1]$ into $N_t$ subintervals:
\begin{equation}
	0 = t_0 < t_1 < t_2 < \cdots < t_{N_t-1} < t_{N_t} = 1,
\end{equation}
with $\tau = \frac{1}{N_t}$ and $t_j = j\tau $.


To incorporate  the Neumann boundary condition \eqref{equ:nbc}, we adopt the ghost point method \cite{Leveque2007}, and define the discrete difference  operator $d_{tt}$ as 
\begin{equation}
	d_{tt}u_j(\mathbf{x})  = 
	\left \{
	\begin{array}{ll}
		\frac{2u_j(\mathbf{x}) -  2u_{j+1}(\mathbf{x})}{\tau^2}  & j =  0, \\[4pt]
		\frac{-u_{j-1}(\mathbf{x})+2u_{j}(\mathbf{x}) - u_{j+1}(\mathbf{x})}{\tau^2} & 1 \le j \le N_t-1, \\[4pt]
		\frac{2u_{j}(\mathbf{x}) -  2u_{j-1}(\mathbf{x})}{\tau^2} & j = N_t.
	\end{array}
	\right.
\end{equation}
Additionally, we  define the auxiliary  right hand side function  as
\begin{equation}
	\hat{u}_j(\mathbf{x}) = 
	\left \{
	\begin{array}{ll}
		\frac{2u_0(\mathbf{x})}{\Delta t}   & j =  0, \\[4pt]
		0 & 1 \le j \le N_t-1, \\[4pt]
		\frac{2u_1(\mathbf{x})}{\Delta t} & j = N_t.
	\end{array}
	\right.
\end{equation}

The semi-discretization  of time-space  equation \eqref{equ:tsharmonic} is 
\begin{equation}\label{equ:semidiscretization}
	-d_{tt} u_j(\mathbf{x}) -  \Delta_{g}u_j(\mathbf{x}) =  \nabla_{\mathbf{p}} \cdot  \mathbf{f}_j(\mathbf{x}) + \hat{u}_j(\mathbf{x})  \quad \forall 0\le j \le N_t.
\end{equation}

Let $\mathcal{M}_h$ be a polyhedral approximation of $\mathcal{M}$ with planar triangular surface, and $\mathcal{T}_h$ be the mesh associated with $\mathcal{M}_h$, where $h = \max_{T\in\mathcal{T}_h}\mbox{diam}(T)$ is the maximum diameter. We focus on the continuous linear finite element space defined on $\mathcal{M}_h$, i.e.
\begin{equation}
	S_h = \{v_h \in C^0(\mathcal{M}_h): v_h|_T \in \mathbb{P}_1(T), \quad  \forall T \in \mathcal{M}_h \}. 
\end{equation}

The surface finite element approximation of the semi-discrete problem is to find $u_{j,h} \in S_h$ such that 
\begin{equation}\label{equ:fulldiscrete}
	-(d_{tt}u_{j,h}, v_h)_{\mathcal{M}_h} + (\nabla_{g_h} u_{j,h}, \nabla_{g_h}v_h)_{\mathcal{M}_h} =  (\partial_t c_j + \nabla_{g_h}\cdot \mathbf{d}_j + \hat{u}_j, v_h)_{\mathcal{M}_h} ,
\end{equation}
for $ 0\le j \le N_t$. Here $(\cdot, \cdot)_{\mathcal{M}_h}$ denotes the inner product on $\mathcal{M}_h$, and we decompose the vector-valued function $\mathbf{f}_j$ as $\mathbf{f}_j = (c_j, \mathbf{d}_j)$, where $c_j$ is the contribution from the temporal component and $\mathbf{d}_j$ is from the spatial component.  In the above equation, $\nabla_{g_h}$ means the discrete surface gradient operator on the discrete surface $\mathcal{M}_h$ with metric tensor $g_h$.

\subsection{Gradient enhanced ADMM} \label{ssec:gre}
The key difference between the standard time-space  equation and the equation \eqref{equ:timespace} in ALG2 is that the right hand side function $\mathbf{f}$ is the intermediate iterative solution, which belongs to the discrete function spaces.  A further looking at Step 3 of ALG2, we find that $\mathbf{f}$ actually involves the gradients of the numerical solutions. It is well-known from numerical analysis of partial differential equations that the accuracy of gradient of numerical solutions will be typically one-order less than the accuracy of the solution itself. Going from Step 3 to Step 1, we need to take a further divergence to obtain real right hand side functions. Theoretically, the accuracy of $\nabla_{\mathbf{p}}\cdot \mathbf{f}$ is of zeroth order when the divergence operator applies to $\mathbf{f}$ which results from gradients of standard linear finite element/difference solutions.   Also, we notice that the gradient of $u$ is a piece-wise constant. To make the divergence $\nabla_{\mathbf{p}}\cdot \mathbf{f}$ well-defined, we introduce an interpolation between piece-wise constant function and piece-wise linear polynomial function.


In this subsection,  we address the above issues by introducing a gradient enhanced algorithm. For such purpose, we introduce the time and space gradient recovery operators.  Denote $S_{\tau}$ to be the continuous linear finite element space on the time partition, i.e. consists of temporal piece-wise linear functions.
Let $G_{\tau}:S_{\tau}\rightarrow S_{\tau} $ be the polynomial preserving  recovery (PPR) operator defined in the Subsection \ref{ssec:ppr}. 
It is worth to point out that $G_{\tau}$ is consistent with the central finite difference scheme on the interior node points and consistent with one side second-order finite difference scheme on boundary node.  Let $\{\phi_0(t), \phi_1(t), \cdots, \phi_{N_t}(t)\}$ be the set of nodal basis functions for $V_{\tau}$ and $\boldsymbol{\phi} = (\phi_0(t), \phi_1(t), \cdots, \phi_{N_t}(t))^T$. For every function  in $v_{\tau}\in S_{\tau}$, we rewrite it as $v_{\tau} = \mbox{V}^T\boldsymbol{\phi}$. The recovered gradient $G_{\tau}v_{\tau}$ can be expressed as
\begin{equation}
	G_{\tau}v_{\tau} = (\mathbf{B_t}  \mbox{V})^T\boldsymbol{\phi}_t,
\end{equation}
where $\mathbf{B_t}$ is the matrix representation for the temporal recovery operator $G_{\tau}$. We denote its discrete form as
\begin{equation}\label{equ:timedf}
	d_t v_j = (\mathbf{B_t} \mbox{V})_j.
\end{equation}

Let $G_h: S_h \rightarrow S_h\times S_h \times S_h$ be the parametric polynomial preserving recovery (PPPR) operator for the surface gradient operator defined in the Subsection \ref{ssec:PPPR}.
If we denote the basis of $S_h$ as $\{\psi_j\}_{j=0}^{N_s}$ and represent an arbitrary function $w_h \in S_h$ as 
$w_h = W^T\boldsymbol{\psi}$, the recovered surface gradient $G_hw_h$ can be expressed as
\begin{equation}
	G_{h}v_h = \left( (\mathbf{B_x}\mbox{V})^T\boldsymbol{\psi}, (\mathbf{B_y} \mbox{V})^T\boldsymbol{\psi}, 
	(\mathbf{B_z} \mbox{V})^T\boldsymbol{\psi} \right)
\end{equation}
as demonstrated in \cite{GuoYangZhu}.  In particular, the differentiation matrix $\mathbf{B}_x, \mathbf{B_y}, \mathbf{B_z}$ are sparse. Once the numerical solution is available, the recovery gradient is just sparse matrix and vector multiplication which can be done quite efficiently. 

\begin{algorithm}
	\caption{Gradient enhanced ALG2}
	\label{alg:grad_ADMM2}
	\begin{algorithmic}[1]
			\REQUIRE{Initial guesses: $\mathbf{q}_{h}^0=(\mathbf{q}_{j,h}^0)_{j=1}^{N_t},\sigma_{h}^1=(\sigma_{j,h}^1)_{j=1}^{N_t}$; Discretization: $N_t, \mathcal{M}_h$.}
			\FOR{$k=1,2,\ldots$}
			\STATE \textbf{Step 1:} 
			Find $u_{j,h}^k$ that solves
			\begin{equation}\label{equ:iter_gea}
				-(d_{tt}u_{j,h}^k, v_h)_{\mathcal{M}_h} + (\nabla_{g_h} u_{j,h}^k, \nabla_{g_h}v_h)_{\mathcal{M}_h} =  (G_{\tau} c_j^k + G_h\cdot \mathbf{d}_j^k + \hat{u}_j, v_h)_{\mathcal{M}_h}, 
			\end{equation}
			for $j=1,\ldots,N_t$, where $(c_j^k,\mathbf{d}_j^k)=\mathbf{f}^k=\sigma_{j,h}^k-r\cdot \mathbf{q}_{j,h}^{k-1}$.
			
			\STATE \textbf{Step 2:} For all $j=1,2,\ldots,N_t$:
			\IF{$G_{\mathbf{p}} u_{j,h}^k+\frac{1}{r}\boldsymbol{\sigma_{j,h}}^k\in A$}
			\STATE $$\mathbf{q}_{j,h}^k=G_{\mathbf{p}} u_{j,h}^k+\frac{1}{r}\boldsymbol{\sigma_{j,h}}^k;$$
			\ELSE
			\STATE{Substitute the roots of \eqref{equ:cubic} into \eqref{equ:projection} to obtain $\mathbf{q}_{j,h}^k$.}
			\ENDIF
			
			\STATE \textbf{Step 3:} 
			\STATE 	Update $\boldsymbol{\sigma_{j,h}}^k$ for some fixed parameter $\alpha_r\in \mathbb{R}^+$
			\begin{equation*}
				\boldsymbol{\sigma_{j,h}}^{k+1}=\boldsymbol{\sigma_{j,h}}^k + \alpha_r (G_{\mathbf{p}} u_{j,h}^k - \mathbf{q}_{j,h}^k). 
			\end{equation*}
			\ENDFOR
			\RETURN{($u^k_{j,h})_{j=1}^{N_t},\ (\mathbf{q}^k_{j,h})_{j=1}^{N_t}$, and $(\sigma^k_{j,h})_{j=1}^{N_t}$.}
	\end{algorithmic}
\end{algorithm}

The proposed gradient enhanced algorithm for \eqref{equ:fulldiscrete} is to find $u_{j,h}$ that satisfies
\begin{equation}\label{equ:gea}
	-(d_{tt}u_{j,h}, v_h)_{\mathcal{M}_h} + (\nabla_{g_h} u_{j,h}, \nabla_{g_h}v_h)_{\mathcal{M}_h} =  (G_{\tau} c_j + G_h\cdot \mathbf{d}_j + \hat{u}_j, v_h)_{\mathcal{M}_h}, 
\end{equation}
for $0\le j \le N_t,$ with  $c_j \in S_h$ and $\mathbf{b}_j\in S_h \times S_h \times S_h$. 
We use the notation $G_{\mathbf{p}} := (G_{\tau}, G_h)$ to denote the temporal-spatial gradient recovery operator. The main idea of the gradient-enhanced ADMM is to implement the gradient/divergence operation appeared in the ADMM steps via the gradient recovery techniques. We summarize it in Algorithm \ref{alg:grad_ADMM2}.
	
This provides a new angle to better understand the existing averaging techniques for dealing with the dual variables in the ADMM algorithms for the dynamic optimal transportation, e.g. \cite{LCCS2018}. However, for gradient recovery on surfaces, it has been illustrated in \cite{DongGuo2020,DongGuoGuo2024} that simple averaging may be failed to recover a gradient to high accuracy due to the mesh conditions. It was also shown there that PPPR is the most robust gradient recovery method on surfaces. In the following, we use matrix notation to give a fast implementation of the PPPR method.

Let $\mathbf{S}$ and $\mathbf{M}$ denote the stiffness matrix and mass matrix in \eqref{equ:gea}. Denote the numerical solution $u_{j,h}= \mbox{U}_j^T\boldsymbol{\psi}$ and $\mbox{U} = (\mbox{U}_0^T, \cdots, \mbox{U}_{N_s}^T)^T$. Then, we reformulate \eqref{equ:gea} in the matrix form as
\begin{equation}\label{equ:linsys}
	\mathbf{A}\mbox{U} = \mbox{F},  
\end{equation}
where 
\begin{equation}
	\mathbf{A} = \mathbf{E}\otimes \mathbf{M} + \mathbf{I}\otimes \mathbf{S}
\end{equation}
with $\mathbf{E}$ being the matrix corresponding to the second-order finite difference scheme, and $\mathbf{I}$ is the $(N_t+1)\times(N_t+1)$ identity matrix, while $\otimes$ is the Kronecker product.

To compute the right hand side,  the matrix representation of the gradient recovery operator suggests a more efficient implementation leveraging sparse matrix and vector multiplication. Without loss of generality, assume 
$c_j = \mbox{C}_j^T\boldsymbol{\psi}$ and $\mathbf{d}_j = ((\mbox{D}_j^x)^T\boldsymbol{\psi}, (\mbox{D}_j^y)^T\boldsymbol{\psi}, (\mbox{D}_j^z)^T\boldsymbol{\psi})$. Define the matrix $\mathbf{C} = (\mbox{C}_0, \mbox{C}_1, \cdots, \mbox{C}_{N_t})$ and $\mathbf{D}_* = (\mbox{D}_0^*, \mbox{D}_1^*, \cdots, \mbox{D}_{N_t}^*)$ for $* = x, y, z$.  
The time-space recovered divergence is given by
\begin{equation}
	\mathbf{C}\mathbf{B_t} + \mathbf{B}_x\mathbf{D_x} + \mathbf{B_y}\mathbf{D_y} + \mathbf{B_z}\mathbf{D_Z}.
\end{equation}
The matrix counterpart of $\mbox{F}$ can be represented as 
\begin{equation}
	\mathbf{F} = \mathbf{M}(\mathbf{C}\mathbf{B_t} + \mathbf{B}_x\mathbf{D_x} + \mathbf{B_y}\mathbf{D_y} + \mathbf{B_z}\mathbf{D_z})\mathbf{I}
\end{equation}
plus the modification to incorporate the boundary condition:
\begin{equation}
	\mathbf{F}(:,1) = \mathbf{F}(:,1) + \mathbf{M}\hat{\mbox{U}}_0, \quad
	\mathbf{F}(:,N_t+1) = \mathbf{F}(:,N_t+1) + \mathbf{M}\hat{\mbox{U}}_{N_t+1}.
\end{equation}
Then the right hand side is $\mbox{F} = \mathbf{F}(:)$.

%
%

\subsection{Fast solver of linear system}\label{ssec:fast}
In this subsection, we introduce a fast solver for the large linear system \eqref{equ:linsys}.
The idea is to decouple the time discretization using spectral decomposition and pre-compute the LU decomposition of decoupled linear systems which is similar to the pre-computation idea in \cite{LCCS2018}.

The semidiscretization \eqref{equ:semidiscretization} can be reformulated in the following vector form 
\begin{equation}\label{equ:matequ}
\mathbf{E} \mathbf{u} - \mathbf{I} \Delta \mathbf{u} = \mathbf{r}, 
\end{equation}
where $\mathbf{u} = (u_1, u_2, \cdots, u_{N_t})^T$ and $\mathbf{r} = (\nabla_{\mathbf{p}} \cdot  \mathbf{f}_1(\mathbf{x}) + \hat{u}_1(\mathbf{x}), \cdots, \nabla_{\mathbf{p}} \cdot  \mathbf{f}_{N_t}(\mathbf{x}) + \hat{u}_{N_t}(\mathbf{x}))^T$.

Let $\mathbf{t} = (t_0, \cdots, N_t)^T$. Denote the eigen pair of  the matrix $\mathbf{E}$ by  $(\gamma_i,  \mathbf{e}_i)$. It is not hard to see that $\gamma_i = \frac{2-2\cos(i\pi\tau)}{\tau^2}$ and $\mathbf{e}_i = \cos(i\pi\mathbf{t})$ for $i = 0, \ldots, \quad N_t$. Notice that $\mathbf{e}_i$ forms a basis of $\mathbb{R}^{N_t}$. Then, we can write $\mathbf{u}$ and $\mathbf{r}$ as
\begin{equation}\label{equ:expand}
\mathbf{u} = \sum_{i=0}^{N_t} w_i \mathbf{e}_i\quad \text{ and } \quad \mathbf{r} =  \sum_{i=0}^{N_t} r_i \mathbf{e}_i.
\end{equation}
Plugging into the matrix equation \eqref{equ:matequ} gives
\begin{equation}\label{equ:coupled}
\sum_{i=0}^{N_t}\gamma_i w_i \mathbf{e}_i + \Delta( \sum_{i=0}^{N_t} w_i \mathbf{e}_i) = \sum_{i=0}^{N_t} r_i \mathbf{e}_i. 
\end{equation}
The linear independence of $\{ \mathbf{e}_i\}$ suggests that we can decouple the equation \eqref{equ:coupled} into $N_t+1$ equations 
\begin{equation}\label{equ:decoupled}
\gamma_i w_i + \Delta w_i = r_i
\end{equation}
for $i = 0, \dots, N_t$.

The decoupled equations \eqref{equ:decoupled} can be solved using the surface finite element method with the recovered gradients in the right hand side as in Subsection \ref{ssec:gre}. The corresponding linear systems
are 
\begin{equation}
(\gamma_i \mathbf{M} + \mathbf{S})\mbox{W}_i = \mbox{R}_i. 
\end{equation}
The key observation is that the matrix can be pre-factorized  as $\gamma_i \mathbf{M} + \mathbf{S} = \mathbf{L}_i\mathbf{U}_i$. 
During the iteration in Algorithm \ref{alg:grad_ADMM2}, the solution of \eqref{equ:decoupled} is realized by a sparse matrix  multiplication with a vector.

Let $\mathbf{R}= (\mbox{R}_0, \cdots, \mbox{R}_{N_t})$. Then, $\mathbf{R}$ can be easily obtained from $\mathbf{F}$ using the following relationship
\begin{equation}
\mathbf{R} = \mathbf{F}\mathbf{H}(\mathbf{H}^T\mathbf{H})^{-1},
\end{equation}
where 
\begin{equation}
\mathbf{H} = (\frac{\mathbf{e}_0}{\|\mathbf{e}_0\|}, \cdots, \frac{\mathbf{e}_{N_t}}{\|\mathbf{e}_{N_t}\|}).
\end{equation}
Again, the inverse of $\mathbf{H}^T\mathbf{H}$ can be pre-computed and reused.

Once we solve \eqref{equ:decoupled} to get $\mathbf{W}= (\mbox{W}_0, \cdots, \mbox{W}_{N_t})$, the solution of the original time-space  equation \eqref{equ:linsys} can be obtained as follows
\begin{equation}
\mathbf{U} = \mathbf{W}\mathbf{H}^T,
\end{equation}
and $\mbox{U} = \mathbf{U}(:)$.

\section{Numerical examples}\label{sec:num}
We present in the following series of numerical examples to demonstrate the performance of the proposed gradient enhanced ADMM algorithm. All the numerical tests are carried out on a 14-inch MacBook Pro Apple M1 Chip with 16GB memory.

\begin{figure}[!h]
\centering
\subfigure[Discrete $L^2$ error\label{fig:square_l2}]{
	\includegraphics[width=0.47\textwidth]{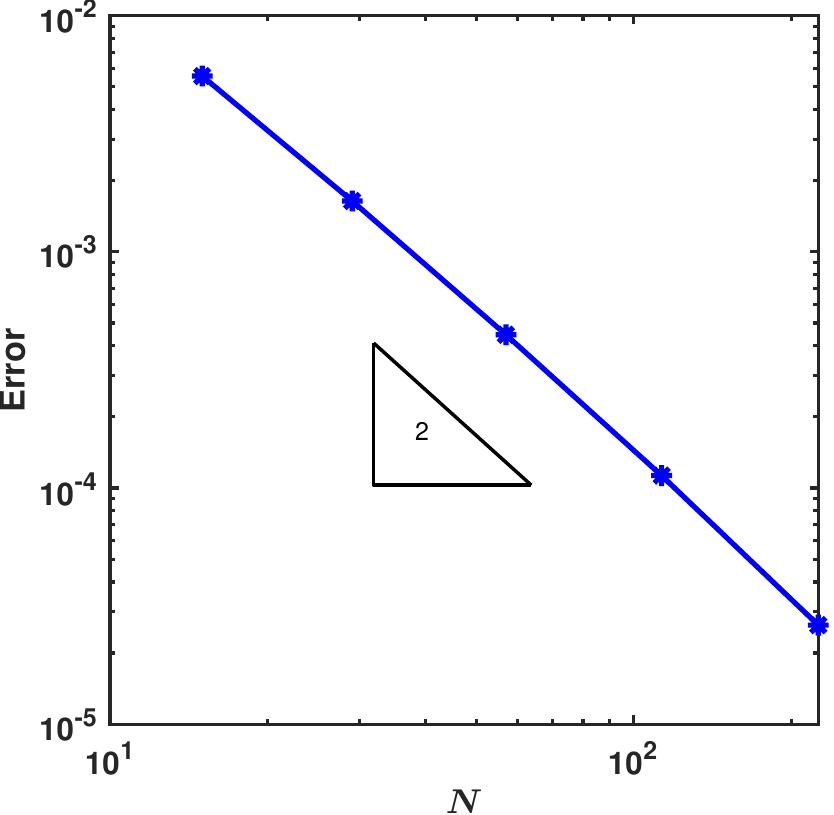}
}
\subfigure[Discrete $L^1$ error\label{fig:square_l1}]{
	\includegraphics[width=0.47\textwidth]{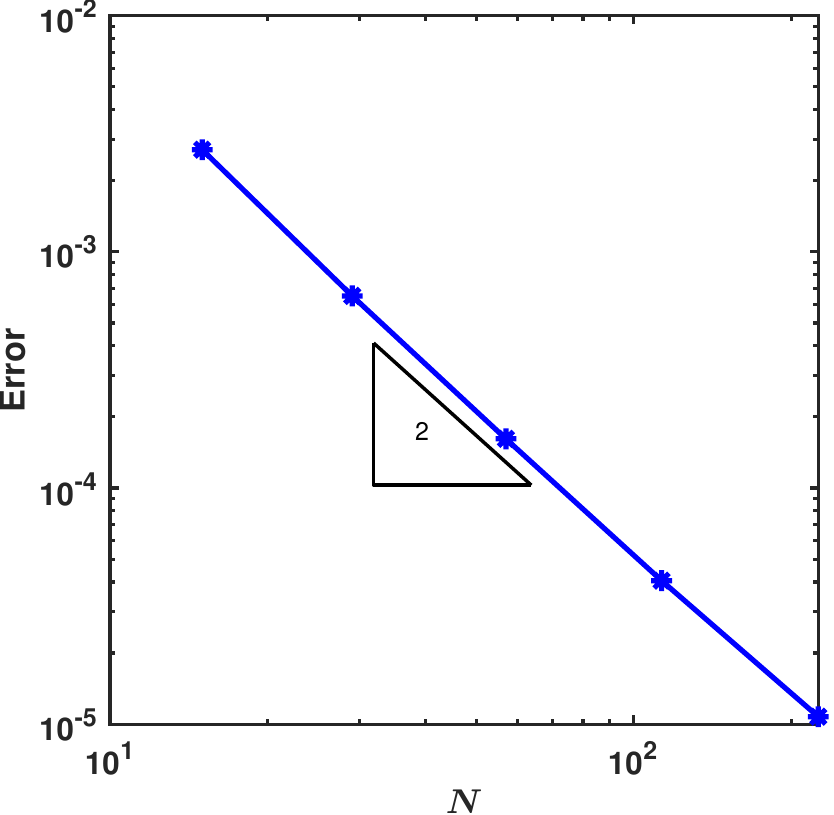}
}
\caption{Numerical error of the optimal transport on the unit square: (a) discrete $L^2$ error; (b) discrete $L^1$ error.}
\label{fig:convergencerate}
\end{figure}

\subsection{Numerical convergence rate} 
We first test the convergence rate of the proposed algorithm using the benchmark example on a 2D flat domain, in which case an exact solution is available. The computational domain is selected as $\Omega = (0,1)\times(0,1)$. The initial and terminal distributions are Gaussian with standard deviation $\sigma = 0.01$ and mean values $\mu_0 = (0.3, 0.3)$ and $\mu_1 = (0.7, 0.7)$ respectively. The ground truth \cite{PeyCut19} is given by
\begin{equation*}
\rho(t, \mathbf{x}) = \frac{1}{\sigma\sqrt{2\pi}}e^{-\frac{1}{2}\left(\frac{\mathbf{x}-\mu(t)}{\sigma}\right)^2}
\end{equation*}
where $\mu(t) = (1-t)\mu_0 + t\mu_1$.

We show the convergence rate of the proposed gradient enhanced ADMM algorithm on uniform meshes.
The uniform meshes are obtained by splitting $\Omega$ into $(N-1)^2$ sub-squares and then dividing each sub-square into 2 triangles. In the numerical test, we choose $N_t = 2N$ with $N = 15, 29, 57, 113, 225$, and $\alpha_r = 0.02$. The number of iterations is set to be $51$. We compute the errors between the computed results $\rho_h$ and its ground truth $\rho$. To quantify the accuracy, we shall measure the error in the following discrete $L^p$ norm:
\begin{equation}
\|v\|_p = \left( \sum_{j=1}^{N_s-1}\tau \|v(\cdot, t_j) \|_{L^p(\Omega)}^p \right)^{\frac{1}{p}}. 
\end{equation}
The numerical errors in discrete norms are displayed in Figure \ref{fig:convergencerate}. From the data in Figure \ref{fig:convergencerate}, it is apparent that the numerical errors in both discrete $L^2$ and $L^1$ norms decay at the optimal rate of $\mathcal{O}(N^{-2})$.

\subsection{Computational efficiency}\label{ssec:cputime} 
In this subsection, we test the efficiency of our algorithms by comparing the CPU time of the proposed algorithms using the fast solver and using standard FDM-sFEM solver. We test the optimal transport (OT) problem on the unit sphere. The initial (terminal) probability distribution $\rho_0$ ($\rho_1$) is the spherical Gaussian \cite{ChaGui2021} with mean $\mu_0 = (0, 0, 1)$ ($\mu_1 = (0, 0, -1)$) and variance $\sigma = 0.1$ at the north (south) pole.


\begin{figure}[!h]
\centering
\subfigure[$t=0$ \label{fig:sphere_1}]
{\includegraphics[width=0.19\textwidth]{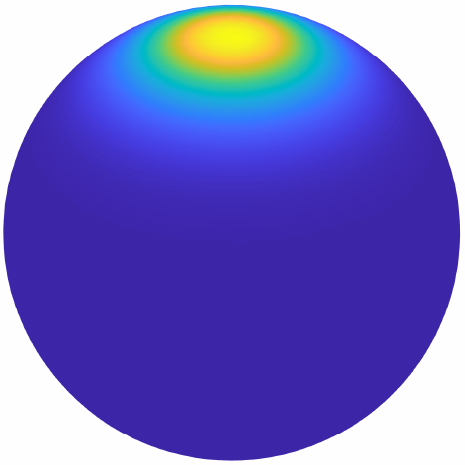}}
\subfigure[$t=0.25$ \label{fig:sphere_2}]
{\includegraphics[width=0.19\textwidth]{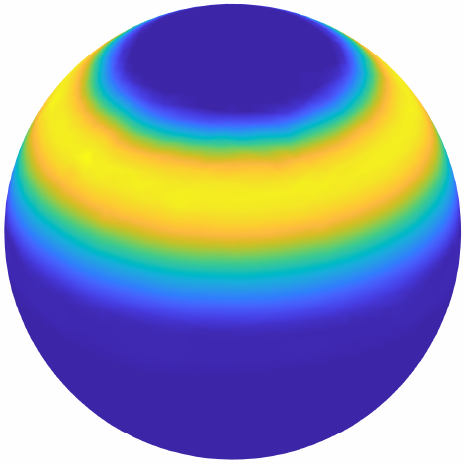}}
\subfigure[$t=0.5$ \label{fig:sphere_3}]
{\includegraphics[width=0.19\textwidth]{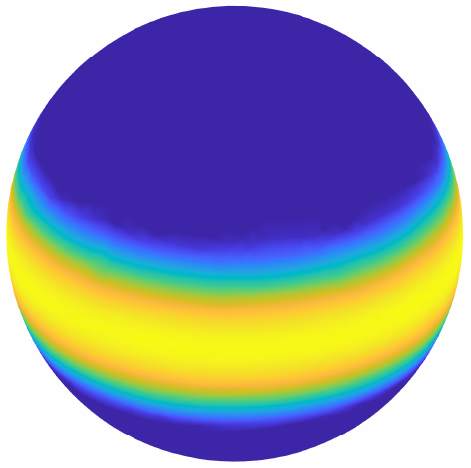}}
\subfigure[$t=0.75$ \label{fig:sphere_4}]
{\includegraphics[width=0.19\textwidth]{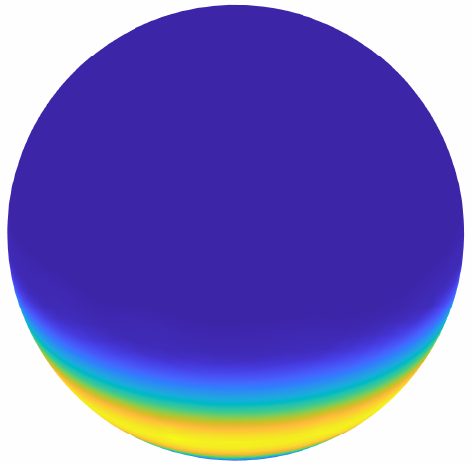}}
\subfigure[$t=1$ \label{fig:sphere_5}]
{\includegraphics[width=0.19\textwidth]{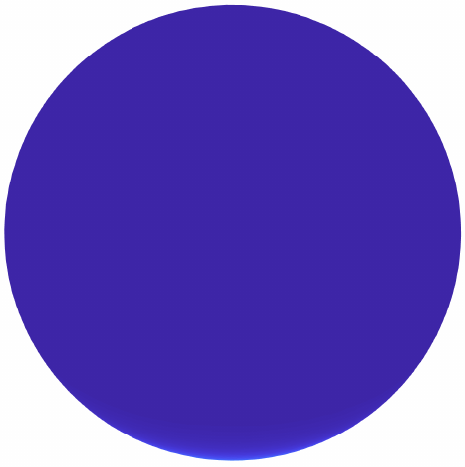}}
\caption{Evolution of mass distribution on the unit sphere. }\label{fig:sphere_test}
\end{figure}

\begin{table}[!h]
\centering
\caption{Comparison of the CPU time for optimal transport (OT) on the unit sphere between ALG2 using the fast solver (F) and ALG2 using built-in solver in MATLAB (C)}\label{tab:comp}
\resizebox{0.75\textwidth}{!}{%
	\begin{tabular}{c|c|cc|cc}
		\hline \multirow{2}{*}{$N_s$} &  \multirow{2}{*}{Iter}  & \multicolumn{2}{c|}{ Time (s) } & \multicolumn{2}{c}{ Time(s)/Iter } \\
		\cline { 3 - 6 }   &  & $\mathrm{F}$  & $\mathrm{C}$ & $\mathrm{F}$ & $\mathrm{C}$  \\
		\hline
		213 & 51 & 0.35 & 6.22 & 6.02e-03 & 1.21e-01 \\ 
		846&  51 & 1.28 & 320 & 2.25e-02 &6.27e+00 \\
		3378&  51 & 8.42 & -- & 1.47e-01 & -- \\
		13506&  51 & 75.28 & -- & 1.30e+00 & -- \\ \hline
	\end{tabular}%
}
\end{table}

In the test, we select $N_t = 2\lfloor \sqrt{N_s}\rfloor$ and 51 iterations. The regularization parameter is set to $\alpha_r = 0.0001$. In Figure \ref{fig:sphere_test}, we plot the evolution of the distribution at several values of $t$. Clearly, the mass distributed with spherical Gaussian at the north pole is transported evenly along every geodesic to the south pole.

In Table \ref{tab:comp}, we compare the CPU times for the gradient-enhanced ALG2 algorithm using the fast solver described in Subection \ref{ssec:fast} and the algorithms using the built-in solver in MATLAB. What stands out in the table is that the algorithm using the built-in solver in MATLAB runs out of memory after 2 uniform refinements of the initial mesh. Closer inspection of values in the table shows that the fast solver achieves more than a hundredfold acceleration per every iteration.

\subsection{Numerical simulation of dynamical optimal transport on general surfaces} 
In this subsection, we are devoted to further investigating the performance of the proposed algorithms for dynamic optimal transport on general surfaces.

\subsubsection{Impact of the augmentation parameter $\alpha_{r}$} 
We investigate the impact of the parameter $\alpha_{r}$ (in Step 3 of Algorithm \ref{alg:ADMM}) on the simulation results. We again consider dynamic optimal transport on the unit sphere, where the initial ($t=0$) and terminal ($t=1$) distributions are spherical Gaussian as plotted in Figure \ref{fig:spheremix}. In the test, we choose $N_s = 3378$ and $N_t = 51$. Figure \ref{fig:spheremix} compares the results obtained from two different values of $\alpha_r$. Interestingly, pseudo transport was observed when $\alpha_r = 1$. In that case, the Gaussian bump is not really transported, and we can observe two Gaussian bumps at the initial and terminal positions.

\begin{figure}[!h]
\centering
\subfigure[$t=0$ \label{fig:spheremix_1}]
{\includegraphics[width=0.19\textwidth]{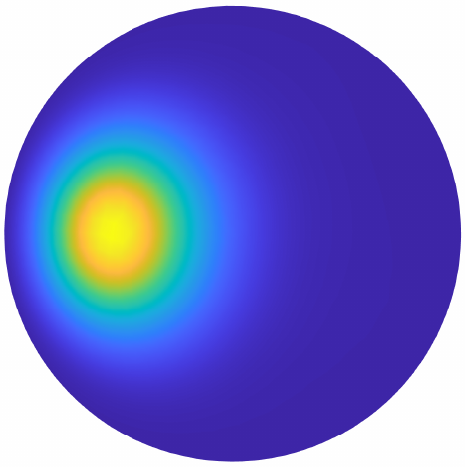}}
\subfigure[$t=0.25$ \label{fig:spheremix_2}]
{\includegraphics[width=0.19\textwidth]{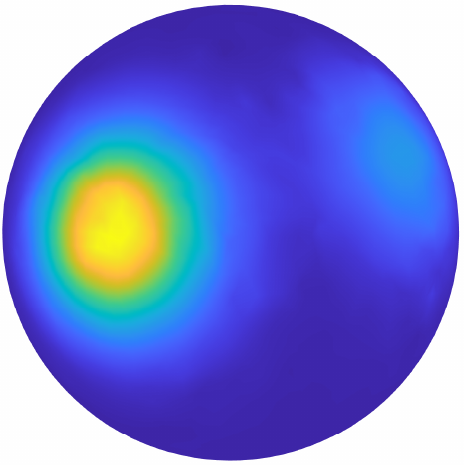}}
\subfigure[$t=0.5$ \label{fig:spheremix_3}]
{\includegraphics[width=0.19\textwidth]{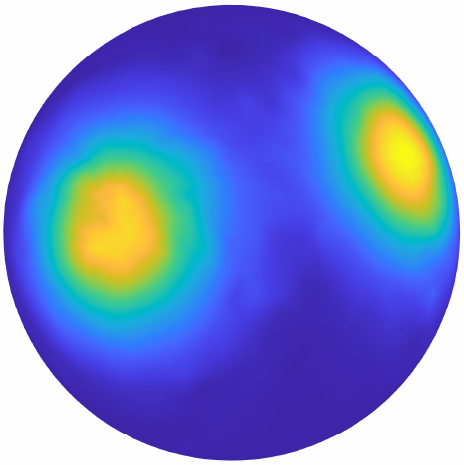}}
\subfigure[$t=0.75$ \label{fig:spheremix_4}]
{\includegraphics[width=0.19\textwidth]{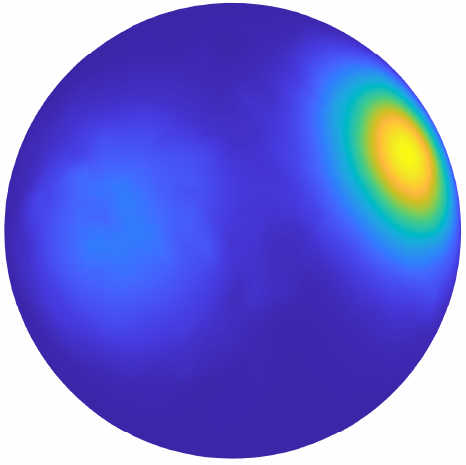}}
\subfigure[$t=1$ \label{fig:spheremix_5}]
{\includegraphics[width=0.19\textwidth]{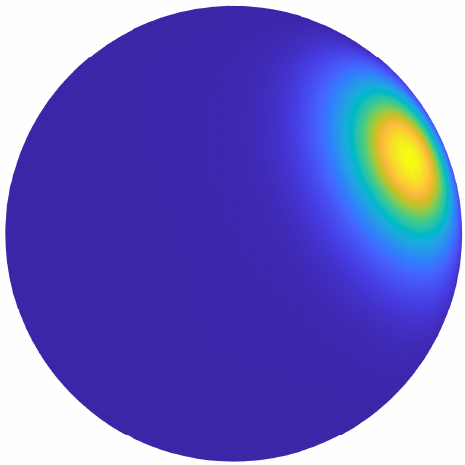}}
\subfigure[$t=0$ \label{fig:spheremix_1_d}]
{\includegraphics[width=0.19\textwidth]{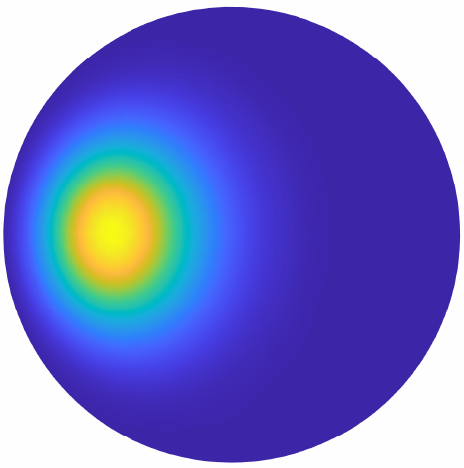}}
\subfigure[$t=0.25$ \label{fig:spheremix_2_d}]
{\includegraphics[width=0.19\textwidth]{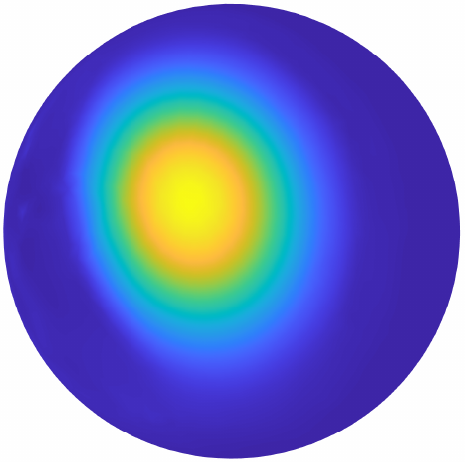}}
\subfigure[$t=0.5$ \label{fig:spheremix_3_d}]
{\includegraphics[width=0.19\textwidth]{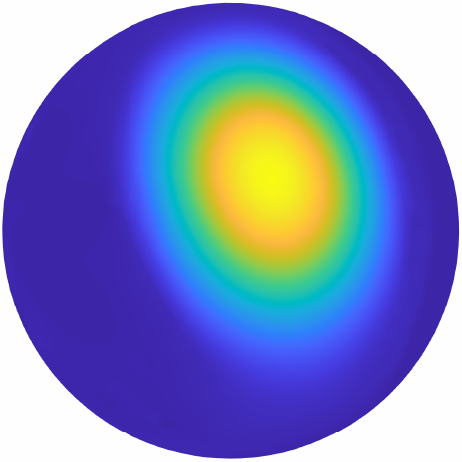}}
\subfigure[$t=0.75$ \label{fig:spheremix_4_d}]
{\includegraphics[width=0.19\textwidth]{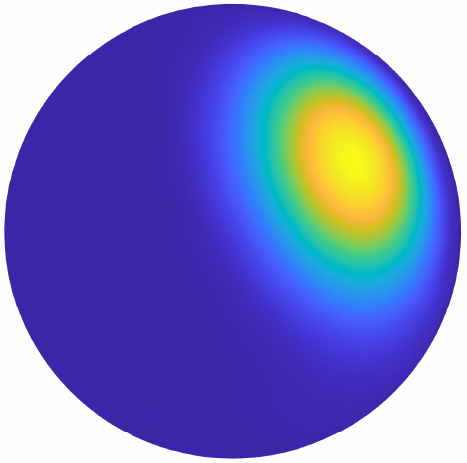}}
\subfigure[$t=1$ \label{fig:spheremix_5_d}]
{\includegraphics[width=0.19\textwidth]{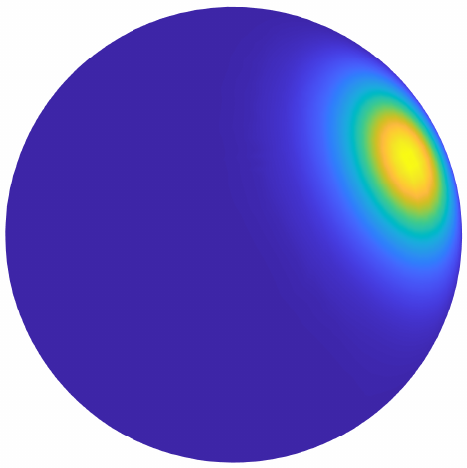}}
\caption{Effect of the regularizing parameter $\alpha_r$. The first row is the numerical result for $\alpha_r = 1$ and the second row is the numerical result of $\alpha_r = 0.0001$.}\label{fig:spheremix}
\end{figure}

\subsubsection{Simulation on surfaces with high curvature and more complex topology}
In this part, we present more numerical results on non-spherical geometry.  Our first example of surfaces with complicated geometry is the  Enzensberger-Stern algebraic surface \cite{DM2016, Guo2020} which is the zero level set of the following  function 
\begin{equation*}
\phi(x)=400\left(x_1^2 x_2^2+x_2^2 x_3^2+x_1^2 x_3^2\right)-\left(1-x_1^2-x_2^2-x_3^2\right)^3-40.
\end{equation*} 
The numerical challenge arises from the high-curvature regions, where the surface meshes can be highly degraded. In our numerical test, we choose the initial (terminal) distribution as the indicator function on the high-curvature angle, illustrated in Figure \ref{fig:stern}. The propagation of indicator functions upward is observable in Figure \ref{fig:stern}. It shows that the mass has been transported along the geodesic on this manifold.

Another example of the dynamical formulation of optimal transport on a genus-1 teapot which is even more complex both in shape and topology, as examined in \cite{LCCS2018}. For this test, we set $\alpha_r = 0.0001$, and the number of iterations is $51$. The simulation results, as depicted in Figure \ref{fig:teapot}, aligning with the results presented in \cite{LCCS2018}, indicate the efficiency of the proposed method also for highly nontrivial surfaces.


\begin{figure}[!h]
\centering
\subfigure[$t=0$\label{fig:stern_1}]
{\includegraphics[width=0.19\textwidth]{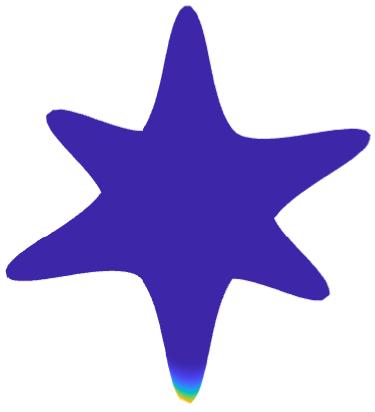}}
\subfigure[$t=0.25$\label{fig:stern_2}]
{\includegraphics[width=0.19\textwidth]{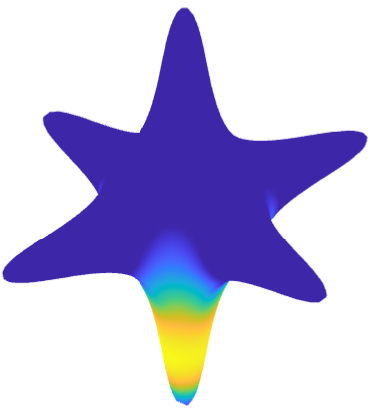}}
\subfigure[$t=0.5$\label{fig:stern_3}]
{\includegraphics[width=0.19\textwidth]{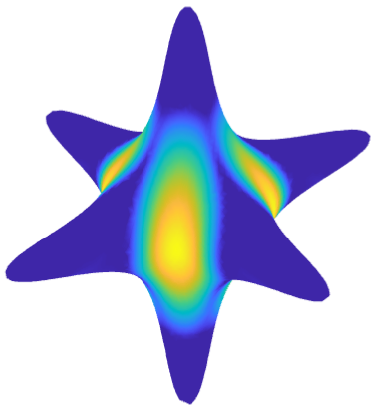}}
\subfigure[$t=0.75$\label{fig:stern_4}]
{\includegraphics[width=0.19\textwidth]{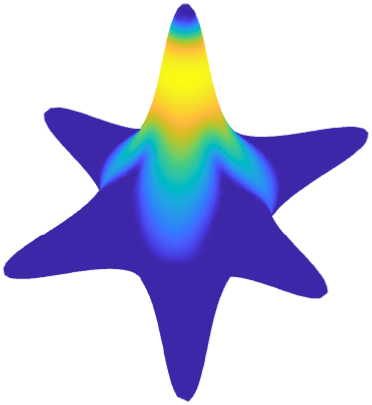}}
\subfigure[$t=1$\label{fig:stern_5}]
{\includegraphics[width=0.19\textwidth]{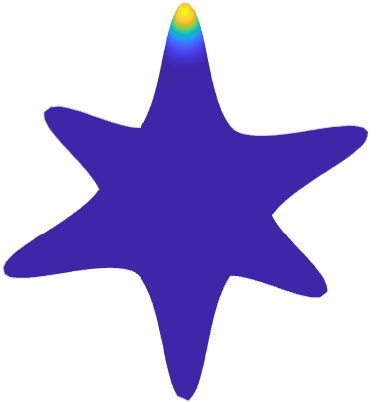}}
\caption{Evolution of mass distribution on Enzensberger-Stern star algebraic surface.}\label{fig:stern}
\end{figure}

\begin{figure}[!h]
\centering
\subfigure[$t=0.25$\label{fig:teapot_1}]
{\includegraphics[width=0.19\textwidth]{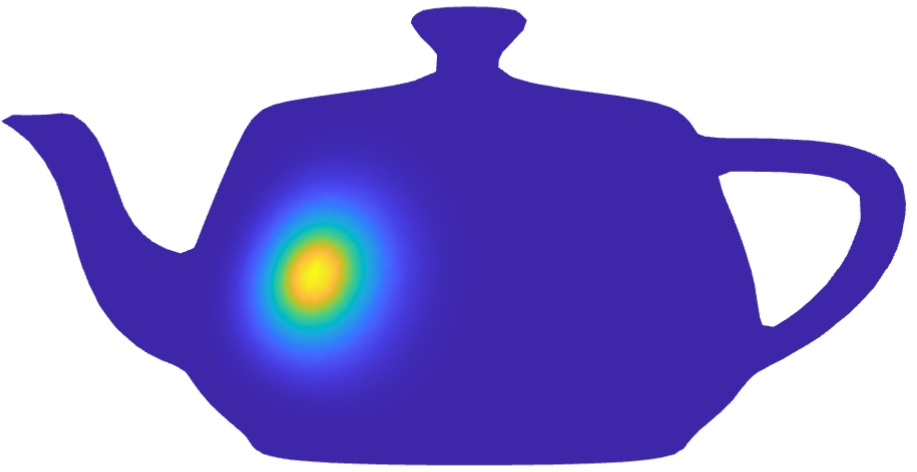}}
\subfigure[$t=0.25$\label{fig:teapot_2}]
{\includegraphics[width=0.19\textwidth]{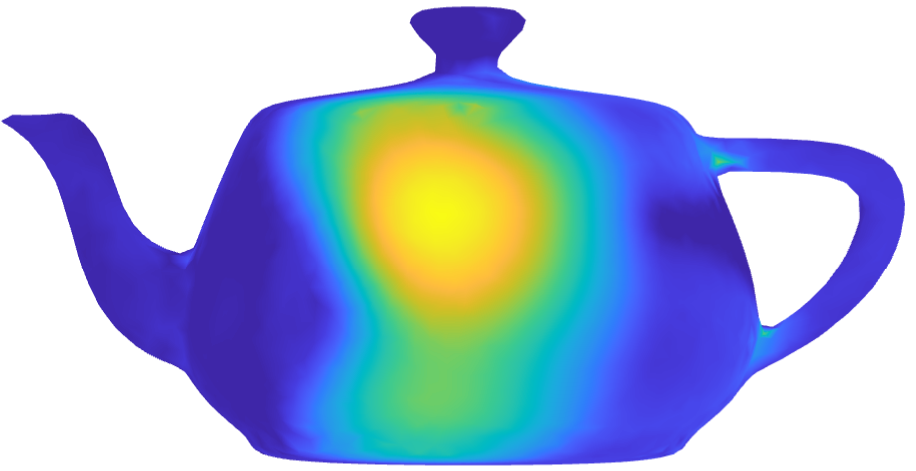}}
\subfigure[$t=0.5$\label{fig:teapot_3}]
{\includegraphics[width=0.19\textwidth]{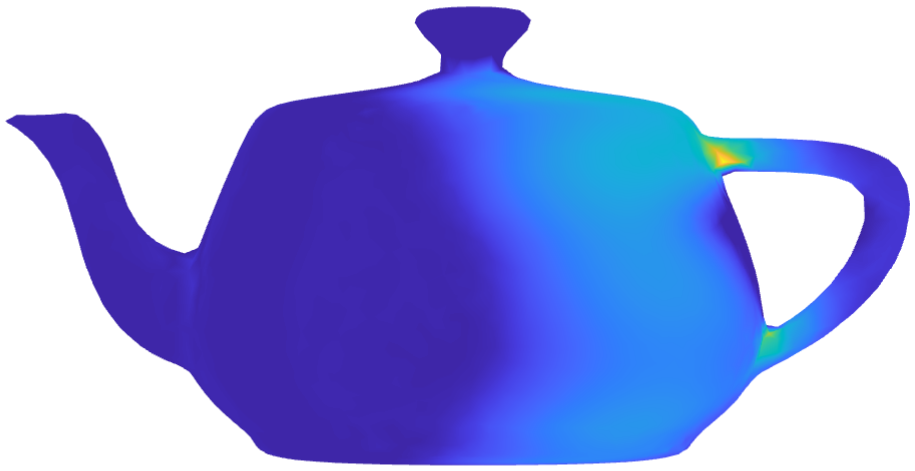}}
\subfigure[$t=0.75$\label{fig:teapot_4}]
{\includegraphics[width=0.19\textwidth]{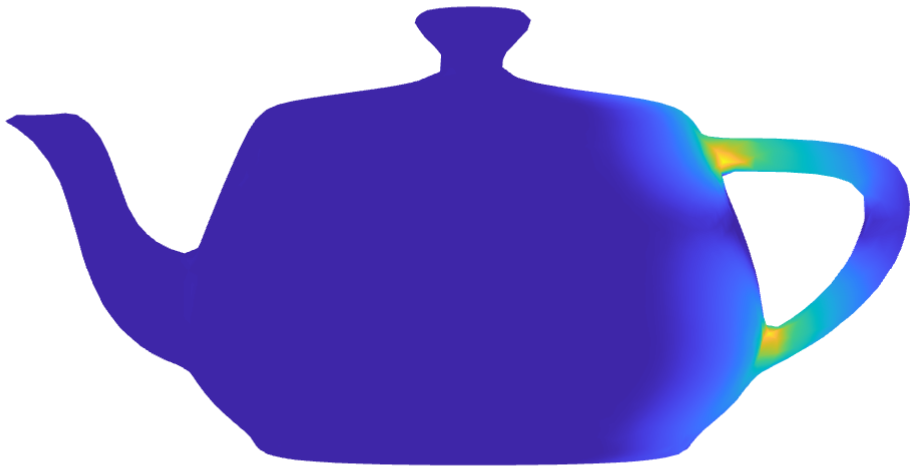}}
\subfigure[$t=1$\label{fig:teapot_5}]
{\includegraphics[width=0.19\textwidth]{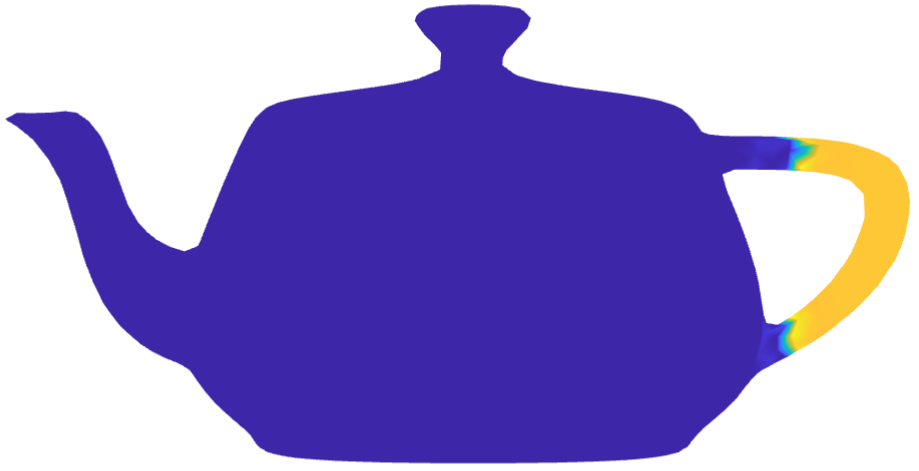}}
\caption{Evolution of mass distribution on a teapot.}\label{fig:teapot}
\end{figure}

\section{Conclusion} 
In this paper, using gradient recovery techniques we have proposed a gradient enhanced ADMM algorithm for numerically realizing the dynamic approach of optimal transportation on surfaces. The incorporation of gradient recovery techniques on surfaces not only  increased the approximation accuracy of numerical gradients, but also unified the presentation of primal and dual variables in the ADMM algorithm using the same grid of a single mesh. Based on that with a spectral decomposition method, we have a fast and robust algorithm for dynamic OT on surfaces. We presented several numerical examples to illustrate the proposed algorithm.


\section*{Acknowledgment}
The work of GD was partially supported by the National Natural Science Foundation of China (NSFC) grant No. 12001194, and the NSF grant of Hunan Province No. 2024JJ5413. The work of HG was partially supported by the Andrew Sisson Fund, Dyason Fellowship, and the Faculty Science Researcher Development Grant of the University of Melbourne. The work of ZS was supported by the NSFC grant No. 92370125.

\bibliographystyle{cas-model2-names}




\end{document}